\documentclass[11pt,a4paper]{article} 

\usepackage[latin1]{inputenc}
\usepackage[twoside,left=3.2cm,right=3.2cm,bottom=3cm,top=2.7cm]{geometry}
\usepackage{amsmath,amsfonts,amsthm,amssymb,mathrsfs,enumerate,comment, amscd}
\usepackage{cite}
\usepackage[all,2cell,dvips]{xy} \SilentMatrices
 \xyoption{pdf}
\usepackage{epsfig,color}

\newtheorem{lemma}{Lemma}[section]
\newtheorem{theorem}[lemma]{Theorem}
\newtheorem{proposition}[lemma]{Proposition}

\newtheorem{corollary}[lemma]{Corollary}
\theoremstyle{definition}
\newtheorem{remark}[lemma]{Remark}
\newtheorem{example}[lemma]{Example}

\numberwithin{equation}{section}
\def\gq{/\!\! /}

\def\RR{{\mathbb R}}

\def\rank{\operatorname{rank}}

\def\c1{\operatorname{c_1}}
\def\c2{\operatorname{c_2}}
\def\Spec{\operatorname{Spec}}

\def\Cox{\operatorname{\mathcal R}}

\def\CC{{\mathbb C}}
\def\ZZ{{\mathbb Z}}

\def\QQ{{\mathbb Q}}
\def\PP{{\mathbb P}}

\def\L{{  L}}

\def\O{{\mathscr O}}

\def\E{{\mathscr E}}

\def\e{{\mathbf e}}

\def\+{\oplus}                   
\def\*{\otimes}                  

\DeclareMathOperator{\Pic}{Pic}
\DeclareMathOperator{\PsAut}{PsAut}

\def\mov{\operatorname{Mov}}

\def\eff{\operatorname{Eff}}
\def\nef{\operatorname{Nef}}

\def\pr{\operatorname{p}}

\newcommand{\pn}[1]
{ \mathchoice
           { \mathbb P^{1}\times {\mathbb P}^{#1} }
          { \mathbb P^{1}\times {\mathbb P}^{#1} }
              { \mathbb P^{1}\times {\mathbb P}^{#1} }
             { \mathbb P^{1}\times {\mathbb P}^{#1} }
         }

\usepackage{sectsty}

\usepackage[small]{titlesec}
\titlelabel{\thetitle\,\,\,\,}

\title{\Large{\bf{Birational geometry of hypersurfaces in products of projective spaces}}}

\author{{{\fontsize{11.5}{1em}\selectfont  John Christian Ottem}}}
\date{}

\begin{document}

\maketitle
\thispagestyle{empty}
\vspace{-0.8cm}

\begin{abstract}
We study the birational properties of hypersurfaces in products of projective spaces. In the case of hypersurfaces in $\PP^m \times \PP^n$, we describe their nef, movable and effective cones and determine when they are Mori dream spaces. Using these results, we give new simple examples of non-Mori dream spaces and analogues of Mumford's example of a strictly nef line bundle which is not ample.
\end{abstract}

\section{Introduction}
Let $X$ be a hypersurface of $\PP^m\times \PP^n$ defined by a bihomogeneous polynomial. If the dimension of $X$ is at least three, the Lefschetz theorem says that the inclusion induces an isomorphism of Picard groups $\Pic(X)\simeq \Pic(\PP^m\times \PP^n)\simeq \ZZ^2$. It is therefore natural to ask how the various cones of divisors of $X$ are related to those of the ambient space $\PP^m\times \PP^n$. In general this relation is not obvious, as examples of Hassett--Lin--Wang \cite{HLW02} and Szendr\~{o}i \cite{Sze03} show that the nef cone of $X$ can be strictly greater than that of the ambient variety. The purpose of this paper is to give a complete picture describing the birational structure of such hypersurfaces. In particular, we compute the cones of effective, movable or nef divisors on $X$ and determine its birational models. In the last section we also consider hypersurfaces in products of more than two projective spaces.

Recall that a normal $\QQ$-factorial projective variety $X$ is a \emph{Mori dream space} if the following three conditions are fulfilled: (i) $\Pic(X)$ is finitely generated; (ii) the nef cone $\nef(X)$ is generated by the classes of finitely many semiample divisors; and 
(iii) there is a finite collection of small $\QQ$-factorial modifications $\phi_i\colon X\dashrightarrow X_i$ such that each $X_i$ satisfies (ii) and the movable cone of $X$ decomposes as $\mov(X)=\bigcup \phi_i^*\nef(X_i)$. 

Mori dream spaces were introduced by Hu and Keel in \cite{HK00} as a class of varieties with good birational geometry properties. For example, the condition of being a Mori dream space is equivalent to having a finitely generated Cox ring \cite[Theorem 2]{HK00} (see section \ref{notations}). Moreover, choosing a presentation for the Cox ring gives an embedding of $X$ into a simplicial toric variety $Y$ such that each of the modifications $\phi_i$ above is induced from a modification of the ambient toric variety $Y$ (see \cite[Proposition 2.11]{HK00}). From this one shows that the Minimal Model Program can be carried out for any divisor and has a combinatorial structure as in the case of toric varieties.

Being a Mori dream space is a relatively strong condition and there are classical examples of varieties that are not. Perhaps the most famous of these is Nagata's counterexample to Hilbert's 14th problem, in which he proves that the blow-up of $\PP^2$ along the base-locus of a general cubic pencil has infinitely many $(-1)$-curves \cite{Nag60}. This blow-up is clearly not a Mori dream space since each of the $(-1)$-curves would require a generator of the Cox ring. The same phenomenon happens for a K3 surface with large Picard number, where one typically expects infinitely many $(-2)$-curves.

There are also other obstructions to being a Mori dream space than the non-polyhedrality of the effective/nef cones. Indeed, there might be integral classes in the boundary of the nef cone which are not semiample. In this paper we will construct concrete examples of this phenomenon even for Picard number 2; In fact `most' hypersurfaces in $\PP^1\times \PP^n$ have an extremal divisor of the nef cone which does not even have an effective multiple, so they are not Mori dream spaces. Thus these hypersurfaces provide simple counterparts to Nagata's examples above.

The following theorem summarizes the geometry of such hypersurfaces:
\begin{theorem}\label{decones}
Let $X$ be a $\QQ$-factorial, normal hypersurface of bidegree $(d,e)$ in $\PP^m\times \PP^n$ of dimension at least three and let $H_i=\pr_i^*\O(1)$. If $m,n\ge 2$, $X$ is a Mori dream space and the Cox ring  is isomorphic to $k[x_0,\ldots,x_m,y_0,\ldots,y_n]/(f)$.  In particular, 
$${\eff(X)}=\mov(X)=\nef(X)=\RR_{\ge 0} H_1+\RR_{\ge 0}H_2.$$

\noindent When $X$ is a general hypersurface in $\PP^1\times \PP^n$, we have the following:

\begin{enumerate}[(i)]
\item If $d= 1$, the second projection realizes $X$ as the blow-up of $\PP^n$ along $\{f_0=f_1=0\}$, and the exceptional divisor is linearly equivalent to $E=eH_2-H_1$, and 
$$
{\eff}(X)=\RR_{\ge 0} H_1+\RR_{\ge 0} E \mbox{ and }\mov(X)=\nef(X)=\RR_{\ge 0}H_1+\RR_{\ge 0}H_2.
$$
\item $1<d< n$.  There is a variety $X^+$ and a small birational modification $\phi:X\dashrightarrow X^+$, which induces a decomposition $\mov(X)=\nef(X)\cup \phi^*\nef(X^+)$. Also, $$
{\eff}(X)=\mov(X)=\RR_{\ge 0} H_1+\RR_{\ge 0} (eH_2-H_1),\mbox{ and }{\nef}(X)=\RR_{\ge 0}H_1+\RR_{\ge 0}H_2.
$$
\item $d=n$. The divisor $eH_2-H_1$ is base-point free and defines a contraction to $\PP^{n-1}$. Also,
$$\eff(X)=\mov(X)=\nef(X)=\RR_{\ge 0} H_1+\RR_{\ge 0} (eH_2-H_1).$$\item If $e=1$, $X$ is a $\PP^{n-1}$-bundle over $\PP^1$.
\end{enumerate}
In these cases $X$ is a Mori dream space and the Cox ring has the following presentation
\begin{equation}\label{coxringpres}
\Cox(X)=k[x_0,x_1,y_0,\ldots,y_n,z_1,\ldots,z_d]/I
\end{equation}where $I=(f_0+x_1z_1,f_1-x_0z_1+x_1z_2,\ldots,f_{d-1}-x_0z_{d-1}+x_1z_d,f_d-x_0z_d)$. 

 For very general hypersurfaces in $\PP^1\times \PP^n$ of degree $(d,e)$ with $d\ge n+1$ and $e\ge 2$ however, $X$ is {\em not} a Mori dream space. Here $$\overline{\eff(X)}=\overline{\mov(X)}=\nef(X)=\RR_{\ge 0} H_1+\RR_{\ge 0} (neH_2-dH_1).$$
but the divisor $neH_2-dH_1$ has no effective multiple.
\end{theorem}

Note that Mori dream hypersurfaces in $\PP^1\times \PP^n$ have bidegrees $(d,e)$ lying in the L-shaped region given by $\{1\le d\le n \mbox{ or } e=1\}$. Hence it is essentially the value of $d$, rather than the anticanonical divisor, that determines whether a general hypersurface of degree $(d,e)$ is a Mori dream space or not. In particular, it is not true that a sufficiently ample hypersurface in a Mori dream space is again a Mori dream space.

In this case that there are only a few bidegrees $(d,e)$ where $X$ is Fano, in which case it is well-known that $X$ is a Mori dream space. The case $(2,n+1)$ corresponds to a Calabi-Yau manifold. On the other hand, for general type varieties, when $K_X$ is ample, most hypersurfaces are not Mori dream spaces, although there are some that are (e.g., bidegree $(n,n+1)$ is a Mori dream space, $(n+1,n+1)$ is not). 

It is also interesting to note that all the cones involved are rational polyhedral for any bidegree.

\section*{Relation to an example of Mumford}
There is a corresponding result for surfaces in $\PP^1\times \PP^2$, but this requires a slightly modified argument, as the Picard group of $X$ might be larger than that of the ambient space. Nevertheless, using the Noether-Lefschetz theorem, we prove the following analogue of Theorem \ref{decones} for surfaces:

\begin{proposition}\label{desurfaces}
Let $X$ be a very general surface in $\PP^1\times \PP^2$ of bidegree $(d,e)$. 
\begin{enumerate}[(i)]
\item If $d=1$, $X$ is the blow-up of $\PP^2$ along the intersection of two very general degree $e$ curves. It is a Mori dream space if and only if $e\le 2$, in which case $X$ is a del Pezzo surface. In the other case, $X$ is a rational surface with infinitely many $(-1)$-curves.
\item If $d=2$, $X$ is a double cover of $\PP^2$ branched along a smooth curve of degree $2e$ and is always a Mori dream space.
\item If $e=1$, $X$ is a Hirzebruch surface. 
\item If $d\ge 3$ and $e\ge 2$, then the effective cone of $X$ is not closed. Hence $X$ is not a Mori dream space.
\end{enumerate}
\end{proposition}

Proposition \ref{desurfaces} gives a simple example of a line bundle $L$ which has positive intersection with every curve $C$, but is not ample. More precisely, a very general surface of bidegree $(3,3)$ in $\PP^1\times \PP^2$ has the line bundle $L=\O(2H_2-H_1)$ which satisfies this condition. Of course here $L$ has top-self-intersection 0. In section 5 we also construct higher dimensional analogues of this. Examples of line bundles with such properties were first constructed by Mumford using certain projective bundles $\PP(\E)$ over curves of genus $\ge 2$. Our geometric construction in Theorem \ref{decones} was inspired by Mumford's example, but we use only projective bundles over elliptic curves.

\section*{The Lefschetz theorem for Mori dream spaces} 
Let $X$ be a $\QQ$-factorial, normal hypersurface of $\PP^m\times \PP^n$ where $m,n\ge 2$. In this case it is easy to show that $X$ is a Mori dream space, by computing the Cox ring of $X$ directly. Indeed, let $D$ be any divisor on $\PP^m\times \PP^n$ and consider the sequence
$$
0\to H^0(\PP^m\times \PP^n,\O(D-X))\to H^0(\PP^m\times \PP^n,\O(D))\to H^0(X,\O_X(D))\to 0
$$This is exact on the right because $H^1(\PP^m\times \PP^n,L)=0$ for any line bundle $L$ when $m,n\ge 2$. This shows that the Cox ring of $X$ is a quotient of that of $\PP^m\times \PP^n$: $$\Cox(X)= k[x_0,\ldots,x_m,y_0,\ldots,y_n]/(f),$$ where the grading is $\deg x_i=H_1$ and $\deg y_j=H_2$. Moreover, the only contractions of $X$ are given by the two projections. 

In particular, this means that the questions mentioned in the introduction are only interesting for hypersurfaces in $\PP^1\times \PP^n$. Here the situation becomes more complicated because the presence of higher cohomology makes it difficult to compute $H^0(X,\O_X(D))$. In particular, we will see that the Cox ring of $X$ is {\em not} a quotient  of that of $\PP^1\times \PP^n$ in these cases.

\bigskip

In general it is an interesting question when a sufficiently ample hypersurface in a Mori dream space is again a Mori dream space. This is not always the case, even for arbitrarily ample hypersurfaces, as shown by Theorem \ref{decones}. In the positive direction, Hausen \cite{Hau08}, Jow \cite{Jow11} and Artebani-Laface \cite{AL12} give criteria for when the Cox ring of the hypersurface is a quotient of that of the ambient variety. A necessary condition for this to hold is that $X$ and $Y$ have isomorphic Picard groups. In \cite{Jow11}, Jow proves that $\Cox(X)\simeq \Cox(Y)/(f)$ for any smooth ample divisor $X$ on $Y$, provided $Y$ is smooth of dimension $\ge 4$, and $I_{\text{irr}}(Y)$ has codimension at least 3 in $\Cox(Y)$. Here $I_{\text{irr}}(Y)$ denotes the so-called \emph{irrelevant ideal} of $Y$, which describes the unstable locus of the action of the Picard torus on $\Cox(Y)$. In the case $Y=\PP^m\times \PP^n$, the irrelevant ideal is given by $I_{\text{irr}}=(x_0,\ldots,x_m)\cap (y_0\ldots,y_n)$, which has codimension at least $3$ if and only if $m,n\ge 2$. Generalizing Jow's result, Artebani and Laface show that the above conclusion holds also under considerably weaker assumptions on $Y$ provided $X$ is ample and general in its linear system  \cite{AL12}.

\section*{Notation}\label{notations} Throughout the paper we will be working over an uncountable algebraically closed field of characteristic 0. 
The main reason for this is that some of the arguments used in section 5 requires working with \emph{very general} hypersurfaces over {\em general} ones. Here the latter will refer to the hypersurface being chosen outside a finite union of closed algebraic subsets of the parameter space, whereas `very general' means outside a countable union of closed algebraic subsets.

 We will also use properties of semistable vector bundles (such as the fact that a symmetric power of a semistable vector bundle is again semistable),  which are known to be false in positive characteristic.  On the other hand it is likely that many of the results in section 2 and 3 can be extended to positive characteristic using the Grothendieck--Lefschetz theorem, which is known to hold in all characteristics.

We let $N^1(X)$ denote the N\'eron-Severi group of $X$, i.e., the $\RR$-vector space of divisors modulo numerical equivalence. For most of the varieties in this paper numerical and linear equivalence coincide, so that $N^1(X)=\Pic(X)\otimes \RR$. Inside $N^1(X)$ we define the effective cone $\eff(X)$ to be the cone of all effective divisors. Similarly, we denote by $\nef(X)$ (resp. $\mov(X)$) the cone of nef divisors (resp. movable divisors). Here we call a divisor $D$ \emph{nef} (resp. \emph{movable}) if $D\cdot C\ge 0$ for every curve $C$ (resp. if the linear system $|mD|$ has no fixed components for $m>0$ sufficiently large). Note that the nef cone is always closed, whereas the other two need not be. 

When $\Pic(X)$ is a free abelian group, the Cox ring of $X$ is defined as the ring $$\Cox(X)=\bigoplus_{\mathbf m\in \ZZ^\rho}H^0(X,\O_X(m_1D_1+\ldots+m_\rho D_\rho))$$ for a chosen basis $D_1,\ldots,D_\rho$ for $\Pic(X)$. As usual, we consider this ring with its $\Pic(X)$-grading. (The Cox ring can more generally be defined using the class group as in \cite{CoxBook}, but this is not necessary for the varieties considered in this paper).

\subsection*{Acknowledgements}Thanks to Burt Totaro for his advice and encouragement. Also thanks to Laurent Gruson, Antonio Laface, Victor Lozovanu, Christian Peskine, Diane MacLagan and  Kenji Oguiso for useful discussions and comments. After this paper was written, we learned that some of the hypersurfaces in Theorem \ref{decones} were considered by Ito in \cite{Ito13}, who shows that they are Mori dream spaces using a different argument.

\section{Mori dream hypersurfaces in $\PP^1\times \PP^n$}

Consider a hypersurface $X\subset\pn{n}$ defined by a bihomogeneous form
\begin{equation}\label{deform}
f=x_0^df_0+x_0^{d-1}x_1f_1+\ldots+x_1^df_d=0
\end{equation}where $x_0,x_1$ coordinates on $\PP^1$ and the $f_i$ are homogenous forms of degree $e$ in the coordinates $y_0,\ldots,y_n$ on $\PP^n$. We will let $\O_{\pn{n}}(a,b)$ denote the line bundle $\pr_1^*\O(a)\otimes \pr_2^*\O(b)$ on $\pn{n}$. We will in this section assume that $n\ge 3$ and that $1<d\le n$. For the remaining cases, see sections 3 and 4. We will also assume that $X$ is general in the sense that it is smooth and that the $f_i$ generate a regular sequence. (However many of the arguments go through under weaker assumptions, e.g., normal and $\QQ$-factorial). In this case, we have by the Grothendieck--Lefschetz theorem that $\Pic(X)= \ZZ H_1\+\ZZ H_2$ where $H_1=\O_{\pn{n}}(1,0)|_X$ and $H_2=\O_{\pn{n}}(0,1)|_X$.

\def\AA{\mathbb A}

The hypersurface $X$ admits an interesting birational map which can be seen if we write $f$ as the determinant of the {\em companion matrix}
\begin{equation}\label{detM}
A=\begin{pmatrix} 
x_1 & 0 &\cdots & 0 &  f_0\\
-x_0 & x_1  & \ddots & \vdots & f_1 \\
0 & \ddots & \ddots &0   & \vdots\\
   & \ddots & -x_0  & x_1 &f_{d-1}\\
0   & \cdots & 0  & -x_0 & f_d\\
\end{pmatrix}
\end{equation}Let $Y\subset \AA=\AA^2\times \AA^{n+1}$ be the affine hypersurface defined by $f=\det A$. Note that if there is a $z=(z_1,\ldots,z_d,1)\in \CC^{d+1}$ with $A\cdot z^t=0$, then also $B\cdot (x_0,x_1,1)^t=0$ where

\begin{equation}\label{yzmatrix}
B=\begin{pmatrix} 
0 & z_1 & f_0\\
-z_1 & z_2 & f_1\\
\vdots & \vdots & \vdots\\
-z_{d-1} & z_d & f_{d-1}\\
-z_d & 0 & f_d\\
\end{pmatrix}
\end{equation}Let $Y^+$ be the subvariety of $\AA^{d}\times \AA^{n+1}$ defined by the maximal minors of $B$. 

Note that for fixed $(x_0,x_1)\in \AA^2-0$, the kernel of $A$ is at most 1-dimensional. Hence we get a well-defined rational map $\psi: Y\dashrightarrow Y^+$ by defining $$\psi(x_0,x_1,y_0,\ldots,y_n)=(z_1,\ldots,z_d,y_0,\ldots,y_n).$$Similarly, given $(z_1,\ldots,z_d,y_0,\ldots,y_n)\in Y^+$ with at least one $z_i$ non-zero, the matrix $B$ also has a kernel which is 1-dimensional, giving a well-defined inverse of $\psi$. So the map $\psi$ is birational.

Everything here is compatible with the various $\CC^*$-actions, so we get a birational map $\phi:X\dashrightarrow X^+$, where $X^+\subset \PP^{d-1}\times \PP^n$ is defined by the minors of $B$. 

The map $\phi$ is a morphism outside the locus where $f_0=\cdots=f_d=0$ in $X$. Indeed, in this case at least one $z_i$ is non-zero, so the corresponding point in $\PP^{d-1}$ is well-defined.  The corresponding statement also holds for $\phi^{-1}$. In particular, when the $f_i$ form a regular sequence, $\phi$ is an isomorphism in codimension $d$ (thus an isomorphism for $d=n$). 

The variety $X^+$ will usually be singular, even for $f_0,\ldots,f_d$ general, but it will follow from the computation below and \cite[Proposition 1.11]{HK00} that the singularities are $\QQ$-factorial and terminal when $X$ is smooth.

To distinguish between $X$ and $X^+$ we use $\O_{X^+}(1,0)$ and $\O_{X^+}(0,1)$ to denote the line bundles on $X^+$ coming from the two projections. From the construction of $\phi$ we have that $$\phi^*(\O_{X^+}(1,0))=eH_2-H_1 \mbox{ and }\phi^*(\O_{X^+}(0,1))=H_2$$ Here the line bundle $\O_{\PP^{d-1}\times \PP^n}(1,0)$ is not big on $X^+$, since it gives the contraction to the lower-dimensional variety $\PP^{d-1}$ (here we are using  $d\le n$). It follows that $eH_2-H_1$ is in the boundary of the effective cone on $X$. From this, we see that the movable cone decomposes as $\nef(X)\cup \nef(X^+)$.  In particular, $X$ is a Mori dream space. 

This decomposition is illustrated in the figure below in the case $2\le d<n$.

\begin{figure}[h!]
  \centering
    \includegraphics[width=0.3\textwidth]{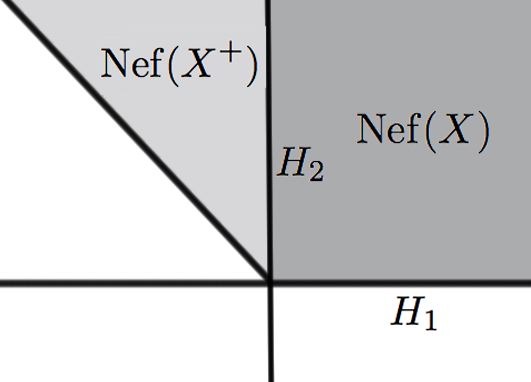}
\label{weight}\label{decompic}
\end{figure}

When $d=n$, $\phi$ is an isomorphism, so the three cones are equal to $\RR_{\ge 0}H_1+\RR_{\ge 0}(eH_2-H_1)$.  The variety $X^+$ and $\phi$ still make sense for $d>n$, but we can not conclude that the effective cone decomposes as above, since $eH_2-H_1$ {\em is} big in this case (cf. section 5).

\bigskip

To compute the Cox ring, we proceed as in \cite{Ott13}. We will assume that $X$ is general (so in particular, $X$ is smooth, and the polynomials $f_0,\ldots,f_d$ form a regular sequence).

We will use induction to write any element of $\Cox(X)$ as a polynomial in the above sections $x_i,y_j,z_k$. Let $D$ be an effective divisor on $X$ and let $L$ be a base-point free line bundle. By a result of Mumford \cite{mumford}, the multiplication map 
$$
H^0(X,\O_X(D-L))\otimes H^0(X,\O_X(L))\to H^0(X,\O_X(D))
$$is surjective provided $H^i(X,\O_X(D-iL))=H^{i}(X,\O_X(D-(i+1)L))=0$ for $i=1,\ldots,m-1$, where $m=h^0(X,\O_X(L))$. Mumford states the result with the assumption that the latter cohomology groups vanish for all $i>0$, but the proof shows that this is not necessary. (In fact, one can give a quick proof of this result by writing out the Koszul complex of $m$ sections generating $H^0(X,\O_X(L))$ and taking its cohomology). Now, if this map is surjective, we see that sections of $\O_X(D)$ are generated by products of sections coming from $\O_X(L)$ and $\O_X(D-L)$, so by induction it follows that $x_i,y_j,z_k$ generate the ring.

To show that the above multiplication map is surjective, suppose $D=aH_1+bH_2$ is an effective line bundle on $X$. If $a\ge 1,b\ge 0$, we may take $L=H_1$ above and note that $H^1(X,\O_X(D-H_1))=0$. When $a=0$, we use instead $L=H_2$ and find that $H^i(X,bH_2)=0$ for all $i=1,\ldots,n-1$ and any $b\ge 0$: This follows by taking cohomology of the exact sequence
$$
0\to \O_{\pn{n}}(-d,b-e)\to \O_{\pn{n}}(0,b)\to \O_{X}(bH_2)\to 0.
$$Similarly, for $a\ge 1,b\ge 0$, the multiplication map $$
H^0(X^+,\O_{X^+}(a-1,b))\otimes H^0(X^+,\O_{X^+}(1,0))\to H^0(X^+,\O_{X^+}(a,b))
$$is surjective. That the above cohomology groups vanish, follows by resolving $\O_{X^+}$ as an $\O_{\PP^{d-1}\times \PP^{n}}$-module, using the Eagon--Northcott complex of $B$ \cite[Appendix B]{Laz04}. It follows that $H^0(X^+,\O_{X^+}(a,b))$ and hence $H^0(X,\O_X(a(eH_2-H_1)+bH_2))$ is spanned by polynomials in $x_i,y_j,z_k$. In all, this means that the Cox ring of $X$ is generated by the sections $x_i,y_j,z_k$. 

Let now $R$ denote the polynomial ring on the right hand side of \eqref{coxringpres}. When the $f_0,\ldots,f_d$ are general, $I$ is a complete intersection and a prime ideal (e.g., it is true for $f_i=y_i^d$ and these are open conditions). In particular, both $R/I$ and $\Cox(X)$ are integral domains. By \cite[Proposition 2.9]{HK00}, the Krull dimension of $\Cox(X)$ equals $\rank \Pic(X)+\dim X=n+2$. Similarly, since $I$ is a complete intersection, the Krull dimension of $R/I$ is $(2+n+1+d)-(d+1)=n+2$. If follows that the surjection $R/I \to \Cox(X)$ is in fact an isomorphism. This completes the proof of the part about $\Cox(X)$ in Theorem \ref{decones}.

\begin{remark}The above birational map $\phi$ has the following interpretation in terms of geometric invariant theory. Consider the $\ZZ^2$-graded ring $R=k[x_0,x_1,y_0,\ldots,y_n,z_1,\ldots,z_d]$ where the grading of the variables is given by columns of the matrix

$$
\left[
\begin{array}{rrrrrrrr}
1 & 1 & 0 & \cdots & 0 & -1 & \cdots &-1\\
0 & 0 & 1 & \cdots & 1 & e & \cdots &e\\
\end{array}
\right]
$$The torus $G=(\CC^*)^2$ acts on $\AA=\Spec R$ via these weights. We wish to study the various GIT quotients $\AA\gq G$. To do this, we consider the trivial line bundle $\L\to \AA$, with coordinate $t$, defined by the embedding $R\subset R[t]$. We extend the action of $G$ to $\L$ by choosing a character $\chi: G\to \CC^*$, and defining for $g\in G$, $g^*(t)=\chi(g)^{-1}\cdot t$.  As shown in \cite{HK00}, the set of semistable points of the action is $\AA-V(B_\chi)$ where $B_\chi$ is the irrelevant ideal of $R$, defined as the radical of the ideal generated by the subring of $R$ with degrees multiples of $\chi$. This defines the GIT quotient of $\AA$ by $G$ associated to $\chi$ as $(\AA-V(B_\chi))/ G$. With our grading, there are essentially three different GIT quotients $Y,Y^+,Z$, corresponding to the characters in the three chambers $\RR_{>0}(\begin{smallmatrix}1\\0\end{smallmatrix})+\RR_{>0}(\begin{smallmatrix}0\\1\end{smallmatrix})$, $\RR_{>0}(\begin{smallmatrix}0\\1\end{smallmatrix})$, $\RR_{>0}(\begin{smallmatrix}1\\0\end{smallmatrix})+\RR_{>0}(\begin{smallmatrix}-1\\e\end{smallmatrix})$ respectively. These correspond to the three irrelevant ideals $B=(x_0,x_1)\cap (y_0,\ldots,y_n,z_1,\ldots,z_d)$, $B^+=(x_0,x_1,y_0,\ldots,y_n)\cap (z_1,\ldots,z_d)$ and $B\cap B^+$
and fit into the following diagram:

\begin{equation}\label{flipdiagram}
\xymatrix{
& Y \ar@{->}[rd]_{\nu} \ar@{->}[ld]_x  \ar@{-->}[rr]^\phi& &Y^+\ar@{->}[ld]^{\nu^+} \ar@{->}[rd]^z\\
\PP^1& & Z & & \PP^{d-1}}
\end{equation}The hypersurface $X$ (resp. $X^+$) can be embedded in $Y$ (resp. $Y^+$) as a complete intersection defined by the $d+1$ equations 
\begin{equation}\label{equationsforX}
f_0+x_1z_1=0,f_1-x_0z_1+x_1z_2=0,\ldots, f_{d-1}-x_0z_{d-1}+x_1z_d=0,f_d-x_0z_d=0.\end{equation}It is straightforward to check that this $\phi$ restricts to the birational map constructed earlier.\end{remark}

\section{Examples}

\subsection{Bidegree $(d,1)$.} Let $X\subset \PP^1\times \PP^n$ be defined by a bihomogeneous form of degree $(d,1)$. The first projection gives $X$ the structure of a $\PP^{n-1}$-bundle over $\PP^1$. In this case, $X$ is a toric variety (hence a Mori dream space) and its birational geometry is well-known \cite{miles}. $X$ has two contractions: One given by the first projection and the other either flipping or contracting depending on $d$.

Let us consider the case where $d\le n$ and $X$ is general. Then it is straightforward to check that the hypersurface is isomorphic to $\PP(\E)$ where $\E= \O_{\PP^1}^{n-d}\oplus \O_{\PP^1}^{d}(1)$. Moreover, the toric variety $Y$ coincides with the projective bundle $\PP(\O^n\+\O(1)^d)$. Under this identification we have $H_1\sim \pi^*\O_{\PP^1}(1)$ and $H_2\sim \O_{\PP(\E)}(1)$, where $\pi:\PP(\E)\to \PP^1$ is the projection map.

The vector bundle $\E$ is generated by the $n+d$ sections $u_{ij}=x_iz_j$ for $i=0,1, j=1,\ldots, d$ and $u_1=z_{d+1}\ldots,u_{n-d}=z_n$ and these sections define an embedding of $X$ inside $\PP^1\times \PP^{n+d-1}$ with defining equations given by the minors of the matrix
$$
\begin{pmatrix}
x_0 &u_{01} & \ldots & u_{0d}\\
x_1 &u_{11} & \ldots & u_{1d}\\
\end{pmatrix}.
$$From this we see that the second projection is birational and the image $Z$ is a cone over the Segre embedding of $\PP^1\times \PP^{d-1}$ in $\PP^{2d-1}$. In fact, from the defining equations we find that $X$ is the blow-up of $Z$ along the ideal $(u_{01},u_{11})$. Blowing up $Z$ along the other ruling gives the other birational model $X^+$.

For $d=2$ and $n=3$, this construction gives the Atiyah flop, where the variety $Z$ above is the quadric cone in $\PP^4$ and $X\to Z$ and $X^+\to Z$ are its two small resolutions.

\subsection{Bidegree $(1,e)$.} If $X$ is defined by a form $f=x_0f_0+x_1f_1$, the second projection contracts the divisor $eH_2-H_1$ and realizes $X$ as a blow-up of $\PP^n$ along the codimension 2 subvariety $Z=\{f_0=f_1=0\}$. The Cox ring of $X$ is isomorphic to $k[x_0,x_1,y_0,\ldots,y_n,z]/(zx_0+f_1,zx_1-f_0)$ and the nef cone is spanned by $H_1$ and $H_2$. On $X$ every movable divisor is nef, as the exceptional divisor is the only possible base-locus of an effective divisor.

\subsection{Bidegree $(2,e)$.} For hypersurfaces $X$ of bidegree $(2,e)$ in $\pn{n}$, the symmetry of the matrices $A$ and $B$ above show that the birational model $X^+$ is actually isomorphic to $X$. We explain this fact as follows. 

Suppose that $X$ is defined by $f=x_0^2f_0+x_0x_1f_1+x_1^2f_2=0$ in $\PP^1\times \PP^n$. The second projection $\pr_2:X\to \PP^n$ is generically 2:1, but it contracts the codimension $2$ locus given by $W=\{f_0=f_1=f_2=0\}$ which is a union of rational curves. Let $\tau:X\to Z$ be the Stein factorization of $\pr_2$. Explicitly, $Z$ is the double cover of $\PP^n$ branched over the divisor given by $D=\{f_1^2-4f_0f_2=0\}\subset \PP^n$. Let $\sigma:Z\to Z$ be the  involution that interchanges the sheets of the double cover. $\sigma$ induces a birational pseudoautomorphism of $X$ defined outside $W$. Using this description, it is easy to show that $\sigma^*H_1+H_1=eH_2$ and $\sigma$ is the $-H_1$-flip of $\tau$. This recovers the decomposition of the movable cone from section 2.

\section{Surfaces in $\pn{2}$}\label{surfaces}
Let $X$ be a very general surface in $\PP^1\times \PP^2$ of bidegree $(d,e)$. Much of the theory from the previous sections can be used to study the birational structure of $X$, but some care must be taken because the Picard group of $X$ might be larger than $\ZZ H_1\+\ZZ H_2$. However, the Noether-Lefschetz theorem of \cite{RS09} says that when $X$ is very general in its linear system and $K_{\pn{2}}\otimes \O_{\pn{2}}(X)$ is globally generated, then we have $\Pic(X)=\Pic(\pn{2})$. This is the case if and only if  $d\ge 2$ and $e\ge 3$. For the remaining cases, we  can proceed by a case-by-case analysis.

$d=1$. Here the situation is drastically different than that of hypersurfaces in higher dimension, because hypersurfaces of bidegree $(1,e)$ hypersurfaces are usually \emph{not} Mori dream spaces. In fact, very general hypersurfaces of bidegree $(1,e)$ can be described as the blow-up of $\PP^2$ along the $e^2$ intersection points of two very general degree $d$ curves. This is known to have infinitely many $(-1)$-curves for $e\ge 3$, so their effective cones of divisors are not rational polyhedral. In these cases the rank of the Picard group is $e^2+1$. For $e=1,2$, they are del Pezzo surfaces and hence Mori dream spaces.

$d=2$. Very general hypersurfaces $X$ of bidegree $(2,e)$ in $\pn{2}$ are Mori dream spaces. Indeed, if $e=2$, $X$ is a del Pezzo surface of degree $4$, which is a Mori dream space with Picard number 6. When $e\ge 3$, the Noether-Lefschetz theorem quoted above gives that $\Pic(X)\simeq \Pic(\pn{2})$. In this case an analysis similar to that in section 2 gives that the nef cone is spanned by $H_1$ and $eH_2-H_1$ and equals the effective cone. Moreover, $\Cox(X)$ has a presentation as in \eqref{coxringpres}.

$e=1$. As before, $X$ is a projective bundle over $\PP^1$, that is, $X$ is a Hirzebruch surface. %

When $d\ge 3$ and $e\ge 2$, a very general surface of bidegree $(d,e)$ is not a Mori dream space. We postpone the proof of this claim to  the next section.

\section{Non-Mori dream space hypersurfaces}\label{nonMDS}

In this section we give examples of bidegree $(d,e)$ hypersurfaces in $\PP^1\times \PP^n$ which do not have a closed effective cone and hence are not Mori dream spaces. From the the previous sections we may restrict to the cases where $d\ge n+1$ and $e\ge 2$.

Consider a `hypersurface' $C$ of bidegree $(2,2)$ in $\PP^1\times \PP^1$. $C$ is an elliptic curve, so $\Pic(C)$ is too big for $C$ to be a Mori dream space. However, it still makes sense to ask whether the subalgebra of $\Cox(C)$ given by 
$$
\Cox(H_1,H_2)=\bigoplus_{a,b}H^0(C,aH_1+bH_2)
$$is finitely generated. It turns out that it is not, at least for $C$ very general in its linear system. Here is the reason: Pick two degree two line bundles $D_1,D_2$ on the elliptic curve $C$ such that $D_1-D_2$ represents a non-torsion point on $\Pic^0(C)$. Then $D_1$ and $D_2$ determine two morphisms $f,g:C\to \PP^1$, hence a morphism $F=(f\times g):C\to \PP^1\times \PP^1$. This is an embedding and the image has bidegree $(2,2)$. However, the line bundle $L=H_2-H_1$ has no effective multiple by our choice of $D_1,D_2$. However, the line bundle $mL+H_1$  has positive degree and so {\em is} effective for any $m\ge 0$. Hence $\Cox(H_1,H_2)$ is not finitely generated.
 
 For the proof of the last part of Theorem \ref{decones}, we will use a variation on this idea. We will consider a projective bundle $Y=\PP(\E)$ of a semistable vector bundle over an elliptic curve $C$ and construct a generically finite morphism from $Y$ to a bidegree $(d,e)$-hypersurface $X_0$ in $\PP^1\times \PP^n$. We will do this in a way, so that the line bundle on $\PP^1\times \PP^n$ restricting to an extremal ray on a very general hypersurface pulls back to a line bundle with no effective multiple on $Y$. By semicontinuity, this will imply that the extremal ray of the nef cone of a general divisor is not semiample, since it has no effective multiple. Hence a very general bidegree $(d,e)$ hypersurface is not a Mori dream space.

We first need the following lemma which gives a bound for the effective cone of a very general hypersurface in $\pn{n}$. It essentially says that the class $neH_2-dH_1$, which has top-self-intersection 0, is pseudoeffective. 
\begin{lemma}\label{subconesss}
Let $X$ be a hypersurface of bidegree $(d,e)$ in $\PP^1\times \PP^n$. Then the $\eff(X)$ contains the subcone
\begin{equation}\label{constraints}
\RR_{> 0}H_1+\RR_{>0}(neH_2-dH_1)
\end{equation}
\end{lemma}

\begin{proof}
Let $L=\O_X(aH_1+bH_2)$ be a line bundle in \eqref{constraints}. We have $L^n=b^{n-1}(bd+aen)>0$. It suffices to show that $L$ is big in the case $a<0$ and $b>0$. By K\"unneth, $H^i(\O_{\pn{n}}(-x,y))=0$ for all $i>1$ and $x,y>0$. So from the exact sequence $$0\to \O_{\pn{n}}(a-d,b-e)\to \O_{\pn{n}}(a,b) \to \L\to 0$$we see that $H^i(Y,mL)=0$ for $i>1$ for $m$ large. Hence $h^0(X,\O_X(mL))\ge \chi(\O_X(mL))=m^nL^n/n!+\ldots$ which is positive for $m$ large. Hence $L$ is big.\end{proof}
The line bundles in this cone correspond exactly to the line bundles $L$ such that $L^n>0$ and $L^{n-1}\cdot H_1>0$. The two divisors $H_1$ and $neH_2-dH_1$ will in fact turn out to be extremal in the effective cone, so generically there exist no divisors of negative top self-intersection on $X$.

\subsection{Construction of the special hypersurfaces.}
Let $C$ be a smooth elliptic curve. By results of Atiyah \cite{Ati57}, $C$ has a semistable rank $r=n$ vector bundle $\E$ of degree $d>n$. Let $Y=\PP(\E)$ denote the variety of hyperplanes in $\E$ with its projection $\pi:Y\to C$. 

An essential point of the construction is defining two morphisms $f:Y\to \PP^1$ and $g:Y\to \PP^n$, so that the induced map $F=f\times g:Y\to \PP^1\times \PP^n$ is birational onto its image, which is a hypersurface of bidegree $(d,e)$. Choose a degree $e$ morphism $p:C\to \PP^1$ (here we are using the fact that $e\ge 2$). Let $f:Y\to \PP^1$ be the composition $f=p\circ \pi$. The generic fiber of $f:Y\to\PP^1$ consists of $e$ distinct fibers of $\pi$. Let $L_1=f^*\O_{\PP^1}(1)$ and $L_2=\O_Y(1)$. Here $L_1\cdot L_2^{n-1}=e$ and $L_2^n=d$.

Note that the choice of the morphism $p:C\to \PP^1$ amounts to choosing a degree $e$ line bundle on $C$ along with two global sections. In this sense, we may talk about $p$ being `general' and `very general' with respect to these data.

\begin{lemma}
The line bundles $L_1$ and $L_2$ are base-point free and $L_2$ is ample.
\end{lemma}

\begin{proof}
$L_1$ is the pullback of a base-point free divisor on $C$. $L_2$ is base-point free because $\E$ is generated by sections: Indeed, this is true for any semistable vector bundle of degree $d> r(2g(C)-1)=n$. When $\E$ is semistable, any effective divisor on $\PP(\E)$ is nef \cite[1.5.A]{Laz04}. Moreover, $L_2$ is big, since it is nef and $L_2^2=d$, so it lies in the interior of the nef cone and hence is ample. 
\end{proof}

When $\E$ is semistable of degree $d> n(2g-1)=n$, we have $h^1(C,\E)=0$ and so by Riemann-Roch, $h^0(C,\E)=d$. We are assuming that $d\ge n+1$, so a choice of $n+1$ generic sections of $L_2$ defines a finite morphism $g:Y\to \PP^n$ of degree $d$.

\begin{lemma}\label{imagebidegree}
For $p:C\to \PP^1$ general, the image $X_0$ of the morphism $F=f\times g:Y\to \PP^1\times \PP^n$ is a hypersurface of bidegree $(d,e)$.
\end{lemma}

\begin{proof}
First of all, $F$ is finite, since $g$ is finite, and so the image is a (possibly singular) hypersurface in $\PP^1\times \PP^n$. We will show that $F$ has degree one below. Granting this for the moment, it follows that the image $X_0$ has bidegree $(d,e)$. Indeed, note that the projections $\pr_1:Y\to \PP^1$ and $\pr_2:Y\to \PP^n$ factor through $f\times g$ and determine the bidegree uniquely: If $X_0$ has bidegree $(a,b)$, we have $a=X_0\cdot \pr_2^*\O_{\PP^n}(1)^n=L_2^n=d$ and $b=X_0\cdot \pr_2^*\O_{\PP^1}(1)\cdot \pr_2^*\O_{\PP^n}(1)^{n-1}=L_1\cdot L_2^{n-1}=e$.

We now show that $F$ is birational onto its image. Recall that the generic fiber of $f$ consists of $e$ disjoint fibers of $\pi$ and $g$ is finite of degree $d$.  Let $y\in Y$ be a general point and let $y'\in Y$ be a point so that  $y\neq y'$ and $g(y)=g(y')$. First we note that the $y$ and $y'$ lie in different fibers of $\pi$; this is because generically the preimage $g^{-1}(l)$ of a line $l\subset \PP^n$ through $g(y)$ is a section of $\pi$ (because $L_2^{n-1}\cdot \pi^{*}\O_C(p)=1$ for a fiber over $p\in C$). 

Suppose now that $y'\neq y$ is a point in $Y$ so that  $F(y)=F(y')$. By the above, we must have $\pi(y)\neq \pi(y')$. Now, we are choosing the degree $e$ map $p:C\to \PP^1$ generically, so we may assume that $\pi(y)$ and $\pi(y')$ map to different points on $\PP^1$. But $f=p\circ \pi$, so $y$ and $y'$ are separated by $f$, and consequently by $F$, a contradiction. In particular, $F$ is injective in a neighbourhood of $y$, and so it is birational onto its image. \end{proof}

\begin{lemma}\label{noeffectivemultiple}
If the morphism $p:C\to \PP^1$ is very general, the divisor $D=enL_2-dL_1$ on $Y=\PP(\E)$ does not have a positive integral multiple which is effective.
\end{lemma}

\begin{proof}
We need to show that for each $m>0$,
$$
H^0(Y,\O_Y(mD))=H^0(C,\pi_*\O_Y(mD))=H^0(C,S^{enm}\E\otimes \O_C(-dm p^*\O_{\PP^1}(1)))=0.
$$Note that $S^{emn}\E$ is a semistable vector bundle of rank $r={emn+n-1 \choose n-1}$ and degree $edmr$ and so $S^{emn}\E\otimes \O_C(-dm p^*\O_{\PP^1}(1)))$ is semistable of degree 0. If $p$ is very general, then these vector bundles do not have any global sections by the lemma below.
\end{proof}

\begin{lemma}\label{nosections}
Let $\E$ be a semistable vector bundle of rank $r$ and degree $0$ on a curve of positive genus. Then for a line bundle $L$ defining a general point in $Pic^0(C)$, we have $H^0(\E\otimes L)=0$.
\end{lemma}

\begin{proof}
If $\E$ is a line bundle, then the statement holds, since the only effective line bundle of degree 0 is the trivial bundle and $\Pic^0(C)$ has dimension $>0$. So we may assume that rank $\E\ge 2$.

The statement is also true if $\E$ is stable, because in that case so is $\E\otimes L$, and if  $\E \otimes L$ has a section, then $\O$ is a subsheaf, contradicting the stability condition.

If $\E$ is strictly semistable, then the Jordan-H\"older filtration says that there is a semistable subbundle $\E'\subset \E$ such that $\E/\E'$ is a stable vector bundle and both $\E$ and $\E/\E'$ have the same slope as $\E$ (that is, 0). From this we get an exact sequence of degree $0$ vector bundles
$$
0\to \E'\to \E\to \E/\E'\to 0.
$$Tensoring this with $L$ and taking cohomology, the result follows by induction on the rank.
\end{proof}

We are ready to prove the main theorem of this section:

\begin{theorem}
Let $X$ be a very general hypersurface of $ \PP^1\times \PP^n$ of bidegree $(d,e)$ with $d\ge n+1$ and $e\ge 2$. Then the effective cone of $X$ is not closed. In particular, $X$ is not a Mori dream space.
\end{theorem}

\begin{proof}
To prove this it is sufficient by semi-continuity of $\dim H^0$ to exhibit a single hypersurface $X_0$ of bidegree $(d,e)$ such that no multiple of the line bundle $L:=\O_{\pn{n}}(-d,ne)$ restricts to an effective divisor on $X_0$; then the same conclusion holds for a very general deformation of it. Since this line bundle is pseudoeffective on any hypersurface, the result follows.

We will let $X_0$ be the image of $Y=\PP(\E)$ under the morphism $F=f\times g$ defined earlier. By construction, the image $X_0$ is a hypersurface of bidegree $(d,e)$ such that the line bundle $F^*\left(L|_{X_0}\right)=\O_Y(enL_2-dL_1)$ has no effective multiple on $Y$ (Lemma \ref{noeffectivemultiple}). Note that ${X_0}$ is reduced (although it may be singular), hence the natural map $\O_{X_0}\to F_*\O_{Y}$ is injective, and we have for $m\ge 1$
\[
H^0({X_0},\left(mL|_{X_0}\right))\subseteq H^0(X_0,\left(m\L|_{X_0}\right)\otimes F_*\O_Y) = H^0(Y,F^*\left(mL|_{X_0}\right))=0\]Hence no multiple of $L$ is effective on ${X_0}$ and the proof is complete.
\end{proof}

\begin{corollary}
Let $X$ be a very general hypersurface $X\subset \PP^1\times \PP^n$ of bidegree $(d,e)$ with $d\ge n+1$ and $e\ge 2$ (and $d,e\ge 3$ in the case $n=2$). Then 
\begin{equation}\label{effectiveconesss}
\overline \eff(X)=\overline \mov(X)=\nef(X)=\RR_{\ge 0}H_1+\RR_{\ge 0}(neH_2-dH_1).
\end{equation}
\end{corollary}

\begin{proof}
By Lemma \ref{subconesss} we have that the pseudoeffective cone contains the cone on the right hand side of \eqref{effectiveconesss}, and so two cones coincide by the theorem. Moreover, $neH_2-dH_1$ is nef on the special hypersurface $X_0$ used in the proof of the theorem, since it pulls back to a nef divisor on $\PP(\E)$ via a finite surjective morphism. Moreover, on $X_0$, the pseudoeffective cone and the nef cone coincide, so by the argument of \cite[Lemma 4.1]{Mou12}, the same conclusion holds for a very general deformation $X$ of $X_0$.
\end{proof}

\begin{example}In dimension 2, $(3,2)$ is the first bidegree for which a very general hypersurface is not a Mori dream space. This variety is rational surface that is isomorphic to a blow-up of a Hirzebruch surface in 9 general points, so the Picard number is in fact 11.
\end{example}

\begin{example}\label{Mumford}
By the Noether-Lefschetz theorem, a very general surface of bidegree $(3,3)$ in $\PP^1\times \PP^2$ has Picard number 2. By the theorem, the line bundle $L=2H_2-H_1$ is nef, but is not semiample. In fact, $L$ is \emph{strictly nef}, in the sense that $\deg L|_C>0$ for every curve $C$ (since $L\cdot C=0$ implies $C\sim_\QQ L$, and $L$ is not $\QQ$-linearly equivalent to an effective divisor). This gives a simple counterpart of Mumford's example mentioned in the introduction.
\end{example}

\begin{example}
A very general hypersurface of bidegree $(4,2)$ in $\PP^1\times \PP^3$ can be viewed as a quadric surface bundle over $\PP^1$. It is therefore a rational threefold with Picard number 2, which is not a Mori dream space.
\end{example}

\begin{remark}[Relation to a conjecture of Keel]
Consider a surface $S$ defined over the field $k=\overline{\mathbb F_p}$ and a line bundle $L$ on $S$. In \cite{keel} Keel posed the problem whether $L$ pseudoeffective implies that it is $\QQ$-effective, that is, that some multiple of $L$ has a section. This is of course false over $\CC$, as we have seen. However, the proof using projective bundles over an elliptic curve fails over $\overline{\mathbb F_p}$, since every degree 0 line bundle is in fact torsion on $\Pic^0(E)$ (and thus is $\QQ$-effective). This raises the question

\bigskip

\noindent {\bf Question. }Let $k=\overline{\mathbb F_p}$ and consider a smooth hypersurface $S$ in $\PP_k^1\times \PP_k^2$ of large bidegree.  Does every pseudoeffective line bundle on $S$ have an effective multiple?
\end{remark}

\section{Hypersurfaces in products of several projective spaces}

It is not surprising that the picture does not become simpler when considering high-degree hypersurfaces in products of more projective spaces. For example, a variation of the previous argument using projective bundles over elliptic curves produces non-Mori dream hypersurfaces multidegree $(d_1,\ldots,d_k,e)$ in $(\PP^1)^k\times \PP^r$ for $d_i\ge r+1,e\ge 2$. In this section, we remark that for hypersurfaces $X$ in products of projective spaces with more than one $\PP^1$-factor, the situation becomes even more complicated. In particular, we don't expect Mori dream spaces, even for low degree hypersurfaces. The following example, which appears in the work of Kawamata \cite[Example 3.8]{Kaw97}, illustrates this already for a Calabi-Yau threefold.

Consider a smooth hypersurface of tridegree $(2,2,3)$ in $\PP=\PP^1\times \PP^1\times \PP^2$ defined by an equation $$f(x_i,y_i,z_i)=x_0^2f_0+x_0x_1f_1+x_1^2f_2$$ where $f_0,f_1,f_2$ are forms of tridegree $(0,2,3)$.  The projection $(\pr_1\times \pr_3):X\to \PP^1\times \PP^2$ contracts the codimension 2 locus $W=\{f_0=f_1=f_2=0\}$ which is a union of 54 rational curves. Taking the Stein factorization gives a small contraction $\phi:X\to Z$ where $Z$ is the double cover of $\PP^1\times \PP^2$ ramified over the divisor defined by $f_1^2-4f_0f_2=0$.  Note that $Z$ has a natural involution $\sigma':Z\to Z$, which switches the sheets of the covering. This determines a birational pseudoautomorphism $\sigma: X\dashrightarrow X$ defined outside $W$. In terms of $H_1,H_2,H_3$ it is not hard to show that $\sigma^*H_1+H_1=2H_2+3H_3$ and that $\sigma:X\dashrightarrow X$ is the $(-H_1)$-flip of $\phi$.

One can repeat the argument with the other contraction $(\pr_2\times \pr_3):X\to \pn{2}$ to get another pseudoautomorphism $\sigma'$ of $X$. Moreover, $\sigma$ and $\sigma'$ generate an infinite subgroup of the group of pseudoautomorphisms of $X$, $\PsAut(X)$. In fact, also the group $\PsAut(X)^*=\mbox{im}(\PsAut(X)\to GL(N^1(X))$ is infinite. In particular, $X$ is not a Mori dream space, because there are infinitely many non-isomorphic marked small $\QQ$-factorial modifications.

Using essentially the same method, one can show the following result:
\begin{proposition}\label{CY}
Let $X$ be a smooth Calabi-Yau hypersurface of dimension $\ge 3$ in $\PP=(\PP^1)^m \times \PP^{n_1}\times \cdots \times \PP^{n_k}$ where $n_1,\ldots,n_k\ge 2$ and let $H_i=\pr_i^*\O(1)$. Then the nef cone is given by $\nef(X)=\RR_{\ge 0} H_1+\RR_{\ge 0} H_2+\ldots+\RR_{\ge 0} H_{k+m}$. Moreover, the following hold:
\begin{enumerate}[(i)]
\item If $m=0$, then $\eff(X)=\mov(X)=\nef(X)=\eff(\PP)$ and $X$ is a Mori dream space.
\item If $m=1$, then the effective cone is strictly larger than that of $\PP$, and $X$ is a Mori dream space.
\item If $m>1$, then $X$ is not a Mori dream space. In fact the group $\PsAut(X)^*$ is infinite and the movable cone is not rational polyhedral.
\end{enumerate}
\end{proposition}

\begin{proof}
The description of the nef cone follows from a more general theorem of Koll\'ar \cite{Bor89}, which states that if $Y$ is a smooth Fano variety of dimension at least $4$ and $X\in |-K_Y|$ is a smooth divisor, then the inclusion induces an isomorphism of the cones of curves $i_*: \overline{NE}_1(X)\to \overline{NE}_1(Y)$. Taking duals gives that the nef cone is the first quadrant $\sum_{j}\RR_{\ge 0}H_j$ in $N^1(X)$.

The case $(i)$ follows as in the first part of Theorem 1.1.

For $(ii)$, let $\sigma: X\dasharrow X$ be the pseudoautomorphism of $X$ obtained by viewing $X$ as a double cover of $\PP^{n_1}\times \cdots \times \PP^{n_k}$.   It is not hard to show that any divisor $D$ lying in the `coordinate planes' $\e_i^\perp=\sum_{j\neq i}\RR_{\ge 0}H_j$ is not big for any $i>1$. This follows for example by taking cohomology of the ideal sheaf sequence $0\to \O_\PP(D-X)\to \O_\PP(D)\to \O_X(D)\to 0$ and using the vanishing of higher cohomology of line bundles on $\PP$. Hence the only supporting hyperplane of $\nef(X)$ containing big divisors is $e_1^\perp$. This hyperplane is fixed under the involution $\sigma^*$, and so by applying $\sigma^*$ to $\nef(X)$, we see that any divisor in the boundary of $\nef(X)\cup\sigma^*\nef(X)$ is not big. Hence $\eff(X)=\mov(X)=\nef(X)\cup\sigma^*\nef(X)$ and so $X$ is a Mori dream space.

$(iii)$ follows as in the example above, noting that two $\PP^1$-factors give rise to two pseudoautomorphisms $\sigma,\sigma'$ generating an infinite subgroup of $\PsAut(X)^*$. 
\end{proof}

In \cite{CO11}, Cantat and Oguiso give a  detailed description of the cones of effective, movable, and nef divisors on hypersurfaces in $(\PP^1)^m$ of multidegree $(2,\ldots,2)$. In particular, they verify the Morrison-Kawamata cone conjecture for these hypersurfaces, which means that even though the movable cone itself fails to be polyhedral, it has a rational polyhedral fundamental domain under the action of $\PsAut(X)$ on $N^1(X)$. In fact, they show that $\PsAut(X)\simeq \ZZ/2\ZZ *\cdots *\ZZ/2\ZZ$, generated by the birational involutions above and that the fundamental domain is in fact the nef cone of $X$. It is likely that these statements generalize to the hypersurfaces in Proposition \ref{CY}.

\bibliographystyle{plain}

\textsc{Department of Pure Mathematics and Mathematical Statistics, University of Cambridge, Wilberforce Road, Cambridge CB3 0WB, UK}

{{\it Email:} \verb"J.C.Ottem@dpmms.cam.ac.uk"}

\end{document}